\documentclass[11pt]{article}
\usepackage[margin=1in]{geometry}
\usepackage{amsmath,amssymb,amsthm,mathtools}
\usepackage{enumitem}
\usepackage[colorlinks=true,linkcolor=blue,citecolor=blue,urlcolor=blue]{hyperref}

\title{Radial Transform Extremality\\for the Siblings of the Coupon Collector}
\author{Christopher D. Long}
\date{}

\newtheorem{theorem}{Theorem}[section]
\newtheorem{lemma}[theorem]{Lemma}
\newtheorem{proposition}[theorem]{Proposition}
\newtheorem{corollary}[theorem]{Corollary}
\newtheorem{definition}[theorem]{Definition}
\newtheorem{remark}[theorem]{Remark}

\newcommand{\E}{\mathbb E}
\newcommand{\Pbb}{\mathbb P}
\newcommand{\R}{\mathbb R}

\newcommand{\dd}{\,d}

\newcommand{\e}{\mathrm e}

\newcommand{\coef}[2]{[#1]#2}
\newcommand{\ind}[1]{\mathbf 1_{\{#1\}}}
\newcommand{\cB}{\mathcal B}

\allowdisplaybreaks

\begin{document}
\maketitle

\begin{abstract}
In the siblings version of the coupon collector, a main collector stops when every coupon type has appeared once. Duplicates are passed successively to siblings, and $U_j^N$ denotes the number of empty spaces in the $j$th collector's album at the main completion time. We prove finite-$N$ radial transform strengthenings of the uniform-probability extremality principle. For every $N\ge2$, every $j\ge2$, every positive nonuniform probability vector $p$, and the ray $p(\theta)=u+\theta(p-u)$ from the uniform vector $u$, the full probability generating function $\E_{p(\theta)}z^{U_j^N}$ is strictly decreasing in $\theta$ for $z>1$ and strictly increasing in $\theta$ for $0<z<1$. Thus the same full PGF has opposite radial monotonicity on the two sides of $z=1$, the left side giving a radial Laplace-transform order. At the coefficient level, along every nonconstant ray from the uniform vector, uniform probabilities maximize every binomial moment of $U_j^N$, equivalently giving a finite absolutely-monotone/binomial-transform order. The proof of the right-PGF and binomial-moment theorem is exact and finite-dimensional. It uses Poissonization, a marked Poissonized PGF identity, a normalized alternating subset expansion, and a positive-kernel radial derivative formula obtained from a local cumulative-polynomial dissipation lemma. The Laplace-transform theorem follows from a separate Gamma-mixture race representation.
\end{abstract}

\begin{center}
\end{center}

\section{The model and the main theorem}

Fix $N\ge2$ and a positive probability vector
\[
        p=(p_1,\ldots,p_N),\qquad p_i>0,
        \qquad \sum_{i=1}^N p_i=1.
\]
Coupons are sampled independently with replacement. The main collector stops at the first time all $N$ types have appeared at least once. Duplicates are passed to the next collector; if that collector already has that type, the coupon is passed to the next sibling, and so on. Let $U_j^N$ be the number of empty spaces in the $j$th collector's album at the main completion time. Equivalently, $U_j^N$ is the number of coupon types whose total count is less than $j$ at the first-completion time.

Let
\[
        u=(1/N,\ldots,1/N).
\]
The main coefficient theorem is the following radial right-PGF extremality result. Here and throughout, \emph{radial} means along line segments in the simplex that start at the uniform vector, namely paths of the form $\theta\mapsto u+\theta(p-u)$. The complementary left-PGF, or Laplace-transform, direction for $0<z<1$ is proved separately in Theorem~\ref{thm:laplace}.

\begin{theorem}[Radial PGF extremality]\label{thm:main}
For every $N\ge2$, every $j\ge2$, every positive nonuniform probability vector $p$, and every $z>1$, the map
\[
        \theta\mapsto \E_{u+\theta(p-u)}z^{U_j^N}
\]
is strictly decreasing on $(0,1]$. In particular,
\[
        \E_p z^{U_j^N}<\E_u z^{U_j^N}\qquad (z>1,\ p\ne u).
\]
Moreover, for every $1\le m\le N$, the map
\[
        \theta\mapsto \E_{u+\theta(p-u)}\binom{U_j^N}{m}
\]
is strictly decreasing on $(0,1]$. In particular,
\[
        \E_p \binom{U_j^N}{m}\le \E_u \binom{U_j^N}{m}\qquad(m\ge1).
\]
The inequality is strict for $1\le m\le N$ and $p\ne u$; for $m>N$ both sides are zero.
\end{theorem}

The proof is given in Sections~\ref{sec:pgf}--\ref{sec:proof-main}. The central point is that the final coupon is always empty in the $j$th sibling's album. Therefore it is natural to remove that deterministic empty space and study
\[
        C_m^{(j)}(p):=\E_p\binom{U_j^N-1}{m},
        \qquad 0\le m\le N-1.
\]
We prove that each nonconstant coefficient $C_m^{(j)}$, $1\le m\le N-1$, is strictly decreasing along every nonconstant ray from the uniform vector. Since every positive probability vector $p$ is the endpoint of the ray $u+\theta(p-u)$, the radial theorem gives uniform as the global maximizer over the positive simplex for these functionals. The result should not be read as a majorization comparison between arbitrary positive nonuniform vectors.

The classical multiple-cover coupon-collector literature goes back to Newman and Shepp \cite{NewmanShepp}. The siblings model and related brotherhood problems have been studied by explicit generating-function and Poissonization methods; representative references include \cite{AdlerOrenRoss,DoumasPapanicolaouSiblings,FoataHanLass,FoataZeilberger}. Those works give exact formulas and asymptotic information, especially for expectations and limiting regimes. The expectation-level radial extremality theorem, corresponding to the case $m=1$ below, appears in the companion paper \cite{LongSiblings}, which also studies monotonicity in $N$ at uniform and limit laws. Doumas and Spektor \cite{DoumasSpektor} also proved the same expectation-level radial extremality theorem and showed that the expectation functional is not Schur-concave. The contribution here is different in kind: we strengthen expectation-level extremality to finite-$N$ radial transform extremality, including the right-side probability generating function, the left-side Laplace transform, and, more sharply, all binomial moments.

\begin{remark}[Order-theoretic scope]\label{rem:scope}
The right-PGF theorem for $z>1$, the Laplace-transform theorem for $0<z<1$, and the finite absolutely-monotone order corollary below are transform-order statements. We use standard stochastic-order terminology as in \cite{ShakedShanthikumar}. The results do not assert ordinary stochastic order, increasing-convex order, hazard-rate order, likelihood-ratio order, or monotone-likelihood-ratio order. Likewise, the results are radial statements away from the uniform vector; they do not compare two arbitrary positive nonuniform probability vectors. Ordinary stochastic order would require coefficientwise positivity of the tail-kernel polynomial
\[
        -\frac{\partial_\theta \E_{p(\theta)}z^{U_j^N}}{z-1},
\]
not merely the transform inequalities proved here.
\end{remark}

\section{Poissonization and the PGF identity}\label{sec:pgf}

Let $N_i(t)$ be independent Poisson processes with rates $p_i$. The superposed process has rate one, and its labels are iid with law $p$. Thus the continuous-time and discrete coupon-count processes have the same embedded label sequence.

Let
\[
        E_i=\inf\{t:N_i(t)\ge1\},
        \qquad X_1=\max_{1\le i\le N}E_i.
\]
The main collector completes at time $X_1$. At this time coupon $i$ is absent from the $j$th album exactly when
\[
        1\le N_i(X_1)\le j-1.
\]
The coupon attaining the last first arrival has count exactly one at $X_1$, and therefore is empty in the $j$th sibling's album for every $j\ge2$.

Define the truncated Poisson polynomial
\[
        A_j(x):=\sum_{r=1}^{j-1}\frac{x^r}{r!}.
\]
For a non-final coupon observed at time $t$, the probability-weighted contribution of the event ``appeared at least once'', with an extra factor $z$ if it is empty in the $j$th sibling's album, is
\[
        1-\e^{-p_i t}+(z-1)\e^{-p_i t}A_j(p_i t).
\]
Indeed, $1-\e^{-p_i t}$ enforces at least one arrival, and
\[
        \Pbb(1\le N_i(t)\le j-1)=\e^{-p_i t}A_j(p_i t).
\]

\begin{proposition}[PGF identity]\label{prop:pgf-identity}
For every $z\ge0$,
\begin{equation}\label{eq:pgf-identity}
        \E_p z^{U_j^N}
        =z\sum_{k=1}^N\int_0^\infty
        p_k\e^{-p_k t}
        \prod_{i\ne k}
        \left(1-\e^{-p_i t}+(z-1)\e^{-p_i t}A_j(p_i t)\right)\dd t .
\end{equation}
\end{proposition}

\begin{proof}
The first-arrival times $E_1,\ldots,E_N$ are independent continuous random variables, so the final first arrival is almost surely unique. Condition on coupon $k$ being this final first arrival and on its first arrival time being $t$. The density in the $k$ coordinate is $p_k\e^{-p_k t}\dd t$. Coupon $k$ contributes one factor of $z$, since it has appeared exactly once at time $t$.

For $i\ne k$, the required weighted contribution is
\[
        \E\left[z^{\ind{1\le N_i(t)\le j-1}}
        \ind{N_i(t)\ge1}\right]
        =\Pbb(N_i(t)\ge1)+(z-1)\Pbb(1\le N_i(t)\le j-1).
\]
Since $N_i(t)$ is Poisson with mean $p_i t$, this equals
\[
        1-\e^{-p_i t}+(z-1)\e^{-p_i t}A_j(p_i t).
\]
Independence of the non-final Poisson processes gives the product over $i\ne k$. Summing over the possible final coupon $k$ and integrating over $t$ proves the identity.

\end{proof}

\begin{theorem}[Radial Laplace-transform order]\label{thm:laplace}
Fix $N\ge2$, $j\ge2$, and a positive nonuniform probability vector $p$. Let
\[
        p(\theta)=u+\theta(p-u),\qquad 0\le\theta\le1.
\]
For every $0<z<1$, the map
\[
        \theta\mapsto \E_{p(\theta)}z^{U_j^N}
\]
is strictly increasing on $[0,1]$. Equivalently, for every $s>0$,
\[
        \theta\mapsto \E_{p(\theta)}\e^{-sU_j^N}
\]
is strictly increasing on $[0,1]$. Thus, if $0\le\theta_1<\theta_2\le1$, then
\[
        U_{j,p(\theta_2)}^N\le_{\rm Lt}U_{j,p(\theta_1)}^N,
\]
where $X\le_{\rm Lt}Y$ means
\[
        \E \e^{-sX}\ge \E \e^{-sY}\qquad(s>0).
\]
\end{theorem}

\begin{proof}
Fix $0<z<1$ and put $r=1-z$. Also write
\[
        B_j(x):=1+A_j(x)=\sum_{\ell=0}^{j-1}\frac{x^\ell}{\ell!}.
\]
Let $F_m$ denote the distribution function of a $\Gamma(m,1)$ random variable. Then
\[
        F_1(x)=1-\e^{-x},
        \qquad
        F_j(x)=1-\e^{-x}B_j(x).
\]
For the non-final factor in Proposition~\ref{prop:pgf-identity},
\[
\begin{aligned}
        1-\e^{-x}+(z-1)\e^{-x}A_j(x)
        &=1-\e^{-x}(1+rA_j(x)) \\
        &=zF_1(x)+rF_j(x).
\end{aligned}
\]
Define iid marked variables $(M_i,Y_i)$ as follows:
\[
        \Pbb(M_i=1)=z,
        \qquad
        \Pbb(M_i=j)=1-z,
\]
and, conditionally on $M_i=m$, let $Y_i$ have the $\Gamma(m,1)$ law. For the radial probabilities $p_i(\theta)$, set
\[
        T_i(\theta)=\frac{Y_i}{p_i(\theta)},
        \qquad
        K_\theta=\arg\max_{1\le i\le N}T_i(\theta),
\]
with any fixed rule for resolving ties. Ties have probability zero for each fixed $\theta$.

For a fixed $k$, the joint contribution of $M_k=1$ and $T_k(\theta)\in\dd t$ has density $z p_k(\theta)\e^{-p_k(\theta)t}\,\dd t$. Conditional on this value of $T_k(\theta)$, the event that $k$ is the maximizing index requires $T_i(\theta)\le t$ for every $i\ne k$. Hence
\[
\begin{aligned}
        \Pbb(M_{K_\theta}=1)
        &=z\sum_{k=1}^N\int_0^\infty
        p_k(\theta)\e^{-p_k(\theta)t}
        \prod_{i\ne k}\bigl(zF_1(p_i(\theta)t)+rF_j(p_i(\theta)t)\bigr)\,\dd t \\
        &=z\sum_{k=1}^N\int_0^\infty
        p_k(\theta)\e^{-p_k(\theta)t}
        \prod_{i\ne k}
        \left(1-\e^{-p_i(\theta)t}+(z-1)\e^{-p_i(\theta)t}A_j(p_i(\theta)t)\right)\,\dd t.
\end{aligned}
\]
By Proposition~\ref{prop:pgf-identity}, the last expression is $\E_{p(\theta)}z^{U_j^N}$. Thus
\begin{equation}\label{eq:laplace-race-representation}
        \E_{p(\theta)}z^{U_j^N}=\Pbb(M_{K_\theta}=1).
\end{equation}

Next condition on $Y_1,\ldots,Y_N$. The posterior probability that a value $Y=y$ came from the $\Gamma(1,1)$ component is
\[
        \pi_z(y):=\Pbb(M=1\mid Y=y)
        =\frac{z\e^{-y}}{z\e^{-y}+(1-z)\e^{-y}y^{j-1}/(j-1)!}
        =\frac{z}{z+(1-z)y^{j-1}/(j-1)!}.
\]
Since $j\ge2$, the function $\pi_z$ is strictly decreasing on $(0,\infty)$. Moreover, because $K_\theta$ is determined by $Y_1,\ldots,Y_N$,
\[
        \Pbb(M_{K_\theta}=1\mid Y_1,\ldots,Y_N)=\pi_z(Y_{K_\theta}).
\]

We now show that the selected value $Y_{K_\theta}$ is nonincreasing in $\theta$. Fix $0\le\theta_1<\theta_2\le1$, and work off the probability-zero event of endpoint ties. Suppose, for contradiction, that
\[
        Y_{K_{\theta_2}}>Y_{K_{\theta_1}}.
\]
Put $a=K_{\theta_1}$ and $b=K_{\theta_2}$. Then $Y_b>Y_a$. For the pair $a,b$,
\[
        \frac{T_b(\theta)}{T_a(\theta)}
        =\frac{Y_b}{Y_a}\frac{p_a(\theta)}{p_b(\theta)}.
\]
Because $p_i(\theta)=1/N+\theta h_i$,
\[
        \frac{\dd}{\dd\theta}\log\frac{p_a(\theta)}{p_b(\theta)}
        =\frac{h_a}{p_a(\theta)}-\frac{h_b}{p_b(\theta)}
        =\frac{h_a-h_b}{N p_a(\theta)p_b(\theta)},
\]
so the pairwise score ratio $T_b(\theta)/T_a(\theta)$ is monotone in $\theta$. At $\theta=0$, all probabilities are equal, so $Y_b>Y_a$ implies $T_b(0)>T_a(0)$. At $\theta_1$, however, $a$ wins, so $T_b(\theta_1)<T_a(\theta_1)$; at $\theta_2$, $b$ wins, so $T_b(\theta_2)>T_a(\theta_2)$. This would force the monotone pairwise ratio to cross the level one twice, a contradiction. Therefore
\[
        Y_{K_{\theta_2}}\le Y_{K_{\theta_1}}
\]
almost surely. Since $\pi_z$ is decreasing,
\[
        \pi_z(Y_{K_{\theta_2}})\ge \pi_z(Y_{K_{\theta_1}}).
\]
Taking expectations and using \eqref{eq:laplace-race-representation} proves monotonicity of $\theta\mapsto\E_{p(\theta)}z^{U_j^N}$.

It remains to prove strictness. Since $p\ne u$, choose $a,b$ with $h_a>h_b$. Then
\[
        \rho(\theta):=\frac{p_a(\theta)}{p_b(\theta)}
\]
is strictly increasing in $\theta$. Fix $0\le\theta_1<\theta_2\le1$ and choose
\[
        c\in\bigl(\rho(\theta_1),\rho(\theta_2)\bigr),
        \qquad c>1.
\]
Choose positive numbers $y_b$ and $y_a=cy_b$. Since the inequalities defining $c$ are strict, there are open intervals $I_a\ni y_a$ and $I_b\ni y_b$ such that
\[
        \frac{Y_a}{Y_b}>\rho(\theta_1),
        \qquad
        \frac{Y_a}{Y_b}<\rho(\theta_2)
\]
whenever $Y_a\in I_a$ and $Y_b\in I_b$. Thus, among the pair $a,b$, the winner is $a$ at $\theta_1$ and $b$ at $\theta_2$. By shrinking $I_a$ and $I_b$ if necessary, the winning pair scores at the two endpoints are bounded below by a positive constant. Hence we may impose sufficiently small upper bounds on all $Y_i$ with $i\notin\{a,b\}$ so that no other index wins at either endpoint. The joint law of $(Y_1,\ldots,Y_N)$ has a strictly positive density on $(0,\infty)^N$, so this event has positive probability. On it, $a$ wins at $\theta_1$, $b$ wins at $\theta_2$, and, since $c>1$,
\[
        Y_{K_{\theta_2}}<Y_{K_{\theta_1}}.
\]
Since $\pi_z$ is strictly decreasing, the conditional posterior probability $\pi_z(Y_{K_\theta})$ strictly increases on this event and is nondecreasing everywhere. Hence
\[
        \E_{p(\theta_2)}z^{U_j^N}>\E_{p(\theta_1)}z^{U_j^N}.
\]
The equivalence with the Laplace-transform statement follows by setting $z=\e^{-s}$.
\end{proof}

Expanding in powers of $q=z-1$ gives
\[
        \frac{\E_p z^{U_j^N}}{z}
        =\sum_{m=0}^{N-1}q^m C_m^{(j)}(p),
        \qquad C_m^{(j)}(p)=\E_p\binom{U_j^N-1}{m}.
\]
The coefficient formula is immediate from Proposition~\ref{prop:pgf-identity}.

\begin{proposition}[Marked coefficient formula]\label{prop:coeff-formula}
For $0\le m\le N-1$,
\begin{equation}\label{eq:coeff-formula}
\begin{aligned}
C_m^{(j)}(p)
        =&\sum_{k=1}^N\int_0^\infty
        p_k\e^{-p_k t}
        \sum_{\substack{S\subseteq[N]\setminus\{k\}\\ |S|=m}}
        \prod_{i\in S}\e^{-p_i t}A_j(p_i t) \\
        &\hspace{1.3in}\times
        \prod_{\ell\notin S\cup\{k\}}(1-\e^{-p_\ell t})\,\dd t.
\end{aligned}
\end{equation}
\end{proposition}

\section{Normalized subset expansion}\label{sec:subset}

We next convert the positive coefficient formula into an alternating normalized subset formula. This is the form in which radial differentiation is finite-dimensional.

For a nonempty set $B\subseteq[N]$, write
\[
        p_B=\sum_{i\in B}p_i,
        \qquad \alpha_i^B=\frac{p_i}{p_B}\quad (i\in B).
\]
For a finite probability vector $\alpha=(\alpha_i)_{i\in B}$ and $0\le m\le |B|-1$, define
\begin{equation}\label{eq:Phi-def}
        \Phi_{m,j}(\alpha)
        :=\sum_{\substack{S\subseteq B\\ |S|=m}}
        (1-\alpha_S)
        \int_0^\infty \e^{-x}
        \prod_{i\in S}A_j(\alpha_i x)\,\dd x,
        \qquad \alpha_S:=\sum_{i\in S}\alpha_i.
\end{equation}
For $m=0$, the empty product is one and $\Phi_{0,j}(\alpha)=1$.

\begin{proposition}[Normalized alternating expansion]\label{prop:subset-expansion}
For $0\le m\le N-1$,
\begin{equation}\label{eq:subset-expansion}
        C_m^{(j)}(p)
        =\sum_{\substack{B\subseteq[N]\\ |B|\ge m+1}}
        (-1)^{|B|-m-1}
        \Phi_{m,j}\left(\alpha^B\right).
\end{equation}
\end{proposition}

\begin{proof}
Expand
\[
        \prod_{\ell\notin S\cup\{k\}}(1-\e^{-p_\ell t})
        =\sum_{R\subseteq[N]\setminus(S\cup\{k\})}(-1)^{|R|}\e^{-p_Rt}.
\]
For fixed $S,k,R$, put $B=S\cup\{k\}\cup R$. Then $|B|-m-1=|R|$. The exponential rate in the integral is $p_B$, and after the change of variables $x=p_Bt$ the corresponding contribution becomes
\[
        \frac{p_k}{p_B}
        \int_0^\infty \e^{-x}
        \prod_{i\in S}A_j\left(\frac{p_i}{p_B}x\right)\dd x.
\]
Conversely, for fixed $B$ and $S\subseteq B$ with $|S|=m$, the original triples are recovered uniquely by choosing the final coupon $k\in B\setminus S$ and setting $R=B\setminus(S\cup\{k\})$. Hence, for fixed $B$ and $S$, summing over final coupons $k\in B\setminus S$ gives
\[
        \sum_{k\in B\setminus S}\frac{p_k}{p_B}=1-\sum_{i\in S}\frac{p_i}{p_B}=1-\alpha_S^B.
\]
This gives \eqref{eq:subset-expansion}.
\end{proof}

\begin{remark}[The oldest sibling]
When $j=2$, $A_2(x)=x$. Therefore
\[
        \Phi_{m,2}(\alpha)=(m+1)!e_{m+1}(\alpha),
\]
where $e_r$ is the elementary symmetric polynomial of degree $r$. Thus Proposition~\ref{prop:subset-expansion} recovers the clean elementary-symmetric formulas for the oldest sibling.
\end{remark}

\paragraph{Proof architecture.}
The remaining proof has two layers. Section~\ref{sec:local} proves a local algebraic fact about a normalized subset: the cumulative truncated Poisson polynomial has a coefficientwise nonnegative pair-dissipation kernel. This is the only place where the cumulative form of $A_j$ is used in an essential way. Section~\ref{sec:proof-main} then inserts that local identity into the alternating expansion above. The alternating signs are converted into positive Laplace kernels, yielding a radial derivative equal to a negative sum of squared pair differences.

\section{The local cumulative-polynomial dissipation lemma}\label{sec:local}

The next result is the algebraic core. It says that, inside every normalized subset, the cumulative polynomial functional dissipates in pairwise-square form. Exact-count pieces do not have this property in general; the cumulative polynomial $A_j$ is essential.

Put
\[
        J=j-1.
\]
For a finite index set $I$ and variables $w_i$, define
\begin{equation}\label{eq:RJI}
        R_{J,I}(w)
        :=\sum_{\nu\in\{1,\ldots,J\}^{I}}
        \frac{|\nu|!}{\prod_{i\in I}\nu_i!}
        \prod_{i\in I}w_i^{\nu_i},
\end{equation}
with the convention $R_{J,\varnothing}=1$. This polynomial satisfies
\begin{equation}\label{eq:R-integral}
        R_{J,I}(w)=\int_0^\infty \e^{-s}\prod_{i\in I}A_j(w_is)\,\dd s.
\end{equation}

For a finite tuple $\gamma=(\gamma_i)_{i\in I}$ of nonnegative integers, write
\[
        M(\gamma):=\frac{|\gamma|!}{\prod_{i\in I}\gamma_i!}.
\]
If the tuple is displayed explicitly, for example $M(a,b,\eta)$, the same definition is used. We also write
\[
        \cB_J(I):=\{1,\ldots,J\}^{I}.
\]

Fix a set $T$ and two additional variables $x,y$. Write $z_T=\sum_{r\in T}z_r$. Define
\begin{equation}\label{eq:FJT}
\begin{aligned}
        F_{J,T}(x,y;z)
        :=&(1-z_T-x)R_{J,\{x\}\cup T}(x,z)
          +(1-z_T-y)R_{J,\{y\}\cup T}(y,z)  \\
        &\quad +\sum_{q\in T}z_q R_{J,\{x,y\}\cup(T\setminus\{q\})}(x,y,z_{T\setminus\{q\}}).
\end{aligned}
\end{equation}
Finally set
\begin{equation}\label{eq:QJT}
        Q_{J,T}(x,y;z)
        :=-\frac{(\partial_x-\partial_y)F_{J,T}(x,y;z)}{x-y}.
\end{equation}
The numerator is divisible by $x-y$, since $F_{J,T}$ is symmetric in $x,y$.

The proof of coefficientwise positivity is divided into two elementary coefficient lemmas. The first concerns a genuinely two-marked box-truncated multinomial polynomial.

\begin{lemma}[Two-marked coefficient formula]\label{lem:two-marked}
Let $U$ be a finite set and put
\[
        P_{J,U}(x,y;w)
        :=-\frac{(\partial_x-\partial_y)R_{J,\{x,y\}\cup U}(x,y,w)}{x-y}.
\]
Then $P_{J,U}$ has nonnegative coefficients. More precisely, if $\mu\in\cB_J(U)$ and $r,s\ge0$, and if
\[
        h=r+s+2,
\]
then
\begin{equation}\label{eq:two-marked-coeff}
\begin{aligned}
        \coef{x^r y^s w^\mu}{P_{J,U}}
        =&\ \ind{h\le J+1}M(h-1,1,\mu) \\
        &+\ind{r\le J-1}\ind{s\le J-1}\ind{h\ge J+2}
          (h-J)M(J,h-J,\mu).
\end{aligned}
\end{equation}
All other $w$-coefficients are zero; equivalently, if the exponent vector in the $w$-variables is not in $\cB_J(U)$, then the corresponding coefficient of $P_{J,U}$ vanishes.
\end{lemma}

\begin{proof}
Fix $\mu\in\cB_J(U)$. Write
\[
        R_\mu(x,y)=\sum_{a=1}^J\sum_{b=1}^J M(a,b,\mu)x^ay^b.
\]
The coefficient of $x^ay^b w^\mu$ in
\[
        D:=-(\partial_x-\partial_y)R_{J,\{x,y\}\cup U}
\]
is
\begin{equation}\label{eq:Dab-two}
        d_{a,b}=-(a+1)M(a+1,b,\mu)+(b+1)M(a,b+1,\mu),
\end{equation}
where a term is interpreted as zero unless every displayed marked exponent lies in $\{1,\ldots,J\}$. If
\[
        P_{J,U}(x,y;w)=\sum_{r,s\ge0}p_{r,s}(\mu)x^ry^s w^\mu+\cdots,
\]
then the identity $D=(x-y)P_{J,U}$ gives, coefficient by coefficient,
\[
        d_{a,b}=p_{a-1,b}(\mu)-p_{a,b-1}(\mu),
\]
with the convention that $p_{r,s}=0$ if either index is negative. Telescoping this recurrence in the $y$-direction gives
\begin{equation}\label{eq:hockey-two}
        \coef{x^r y^s w^\mu}{P_{J,U}}
        =p_{r,s}(\mu)=\sum_{\ell=0}^s d_{r+1+\ell,s-\ell}.
\end{equation}
Thus the desired coefficient is obtained by summing $d_{a,b}$ along the path
\[
        (a,b)=(r+1,s),(r+2,s-1),\ldots,(r+s+1,0).
\]
Along this path $a+b=r+s+1=h-1$.

At an interior point, where both terms in \eqref{eq:Dab-two} are present, the two terms cancel exactly, since
\[
        (a+1)M(a+1,b,\mu)=(b+1)M(a,b+1,\mu).
\]
Consequently only boundary points of the box $\{1,\ldots,J\}^2$ can contribute. There are three possible boundary contributions.

First, there is a lower $y$-boundary contribution at $b=0$. Then $a=h-1$, and only the second term in \eqref{eq:Dab-two} can be present. It is present exactly when $h-1\le J$, equivalently $h\le J+1$, and its value is
\[
        M(h-1,1,\mu).
\]
This gives the first term in \eqref{eq:two-marked-coeff}.

Second, there is a possible upper $x$-boundary contribution at $a=J$. Then $b=h-1-J$, and the first term in \eqref{eq:Dab-two} is absent because it would require the exponent $J+1$ in the $x$-coordinate. The surviving second term is positive and equals
\[
        (b+1)M(J,b+1,\mu)=(h-J)M(J,h-J,\mu).
\]
This contribution occurs precisely when the path reaches $a=J$ and the surviving $y$-exponent $h-J$ lies in $\{1,\ldots,J\}$. Equivalently,
\[
        r\le J-1,
        \qquad J+1\le h\le 2J.
\]
The case $h=J+1$ is already the lower $y$-boundary case $b=0$ above. Thus the genuinely upper contribution is relevant for $J+2\le h\le2J$.

Third, there is a possible upper $y$-boundary contribution at $b=J$. Then $a=h-1-J$, and the second term in \eqref{eq:Dab-two} is absent because it would require the exponent $J+1$ in the $y$-coordinate. The surviving first term is negative and equals
\[
        -(a+1)M(a+1,J,\mu)=-(h-J)M(h-J,J,\mu).
\]
This contribution occurs precisely when
\[
        s\ge J,
        \qquad J+2\le h\le2J.
\]
It has the same magnitude as the upper $x$-boundary contribution, because
\[
        M(J,h-J,\mu)=M(h-J,J,\mu).
\]
Whenever this negative upper $y$-boundary contribution is present, the positive upper $x$-boundary contribution is also present. Indeed, $s\ge J$ and $h\le2J$ imply
\[
        r=h-s-2\le 2J-J-2=J-2<J-1,
\]
so the path also reaches $a=J$. The two upper-boundary contributions therefore cancel in that case.

It remains only to record when the positive upper $x$-boundary contribution is not canceled. This is exactly the case
\[
        r\le J-1,
        \qquad s\le J-1,
        \qquad h\ge J+2.
\]
The missing upper bound $h\le2J$ is automatic from $r,s\le J-1$, since then $h=r+s+2\le2J$. The uncanceled contribution is
\[
        (h-J)M(J,h-J,\mu),
\]
which is the second term in \eqref{eq:two-marked-coeff}. These are all possible boundary contributions, so the displayed formula follows. It is coefficientwise nonnegative, and the proof is complete.
\end{proof}

The second coefficient lemma isolates the one-marked part and records exactly which lower-boundary terms must be canceled by the two-marked terms.

\begin{lemma}[One-marked coefficient formula]\label{lem:one-marked}
Let $T$ be a finite set, and define
\[
        G_{J,T}(x;z):=(1-z_T-x)R_{J,\{x\}\cup T}(x,z)
\]
and
\[
        E_{J,T}(x,y;z):=-\frac{G_{J,T}'(x;z)-G_{J,T}'(y;z)}{x-y}.
\]
Fix $r,s\ge0$, set $h=r+s+2$, and let $\eta$ be a multi-index on $T$. Then
\begin{equation}\label{eq:one-marked-coeff}
        \coef{x^r y^s z^\eta}{E_{J,T}}
        =U_h(\eta)-L_h(\eta),
\end{equation}
where $U_h(\eta)$ is the nonnegative upper-boundary contribution
\begin{equation}\label{eq:upper-Uh}
\begin{aligned}
        U_h(\eta)
        :=&\ h\,\ind{h=J+1}\ind{\eta\in\cB_J(T)}M(J,\eta) \\
        &+h\sum_{q\in T}
        \ind{h\le J}\ind{\eta_q=J+1}
        \ind{\eta_{T\setminus\{q\}}\in\cB_J(T\setminus\{q\})}
        M(h,\eta-e_q),
\end{aligned}
\end{equation}
and $L_h(\eta)$ is the lower-boundary deficit
\begin{equation}\label{eq:lower-Lh}
        L_h(\eta)
        :=\sum_{q\in T}
        \ind{h\le J}\ind{\eta_q=1}
        \ind{\eta_{T\setminus\{q\}}\in\cB_J(T\setminus\{q\})}
        M(h-1,1,\eta_{T\setminus\{q\}}).
\end{equation}
\end{lemma}

\begin{proof}
Let $g_h(\eta)$ be the coefficient of $x^hz^\eta$ in $G_{J,T}(x;z)$. Since
\[
        -\frac{x^{h-1}-y^{h-1}}{x-y}
        =-\sum_{r+s=h-2}x^r y^s,
\]
the coefficient of $x^ry^sz^\eta$ in $E_{J,T}$ is $-h g_h(\eta)$.

The coefficient $g_h(\eta)$ is
\begin{equation}\label{eq:gh}
\begin{aligned}
        g_h(\eta)=&\ \ind{h\le J}\ind{\eta\in\cB_J(T)}M(h,\eta) \\
        &-\ind{h-1\le J}\ind{\eta\in\cB_J(T)}M(h-1,\eta) \\
        &-\sum_{q\in T}\ind{h\le J}
          \ind{\eta_q\ge2}\ind{\eta-e_q\in\cB_J(T)}M(h,\eta-e_q),
\end{aligned}
\end{equation}
where an indicator suppresses a term whose displayed exponent vector is outside its box. We compare the three lines of \eqref{eq:gh} by the multinomial divergence identity
\begin{equation}\label{eq:divergence}
        M(\gamma)=\sum_{i:\gamma_i>0}M(\gamma-e_i),
\end{equation}
valid for every nonzero tuple $\gamma$ of nonnegative integers. Identity \eqref{eq:divergence} follows at once by dividing
\[
        \sum_i \gamma_i=|\gamma|
\]
by $\prod_i\gamma_i!$ and multiplying by $(|\gamma|-1)!$.

First suppose that $h\le J$ and $\eta\in\cB_J(T)$. Then the first line of \eqref{eq:gh} is an interior current coefficient. Applying \eqref{eq:divergence} to the tuple $(h,\eta)$ gives
\[
        M(h,\eta)=M(h-1,\eta)+\sum_{q\in T}M(h,\eta-e_q),
\]
where the summand with index $q$ is understood to be present only when $\eta_q>0$. The second line of \eqref{eq:gh} subtracts the predecessor obtained by lowering the $x$-coordinate from $h$ to $h-1$. The admissible terms in the third line subtract exactly those $T$-coordinate predecessors for which $\eta_q\ge2$. Thus the only predecessors not subtracted are the lower-boundary predecessors with $\eta_q=1$, and hence
\[
        g_h(\eta)=\sum_{q\in T}\ind{\eta_q=1}M(h,\eta-e_q).
\]
For such a coordinate $q$, the absorbed derivative factor is accounted for by the elementary identity
\begin{equation}\label{eq:absorbed-h-factor}
        hM(h,\eta-e_q)=M(h-1,1,\eta_{T\setminus\{q\}}).
\end{equation}
Therefore
\[
        -h g_h(\eta)
        =-\sum_{q\in T}\ind{\eta_q=1}
          M(h-1,1,\eta_{T\setminus\{q\}}),
\]
which is precisely the lower-boundary deficit in \eqref{eq:lower-Lh} in the present case.

It remains to consider the case in which the current coefficient in the first line of \eqref{eq:gh} is absent. Then no negative interior current has to be balanced, and the coefficient can only come from upper-boundary predecessor terms.

The term from the factor $-x$ is present precisely when $h-1\le J$ and $\eta\in\cB_J(T)$. Since the first line is absent, this forces $h>J$, and therefore $h=J+1$. Its contribution to $-h g_h(\eta)$ is
\[
        hM(J,\eta),
\]
which is the first summand of \eqref{eq:upper-Uh}.

The term from the factor $-z_q$ is present precisely when
\[
        h\le J,\qquad \eta_q\ge2,\qquad
        \eta-e_q\in\cB_J(T).
\]
Since the first line is absent and $h\le J$, the vector $\eta$ itself is not in $\cB_J(T)$. The displayed conditions then force
\[
        \eta_q=J+1,
        \qquad
        \eta_{T\setminus\{q\}}\in\cB_J(T\setminus\{q\}),
\]
and the contribution to $-h g_h(\eta)$ is
\[
        hM(h,\eta-e_q).
\]
Summing over $q\in T$ gives exactly the second summand of \eqref{eq:upper-Uh}. In all remaining cases the three lines of \eqref{eq:gh} are absent, and both sides of \eqref{eq:one-marked-coeff} vanish. This proves \eqref{eq:one-marked-coeff}.
\end{proof}

\begin{lemma}[Local cumulative-polynomial dissipation]\label{lem:local}
For every $J\ge1$ and every finite set $T$, the polynomial $Q_{J,T}$ has nonnegative coefficients. Moreover, $Q_{J,T}$ is not the zero polynomial.
\end{lemma}

\begin{proof}
Using the notation of Lemma~\ref{lem:one-marked}, the definition \eqref{eq:FJT} gives the exact decomposition
\begin{equation}\label{eq:Q-decomp}
        Q_{J,T}(x,y;z)
        =E_{J,T}(x,y;z)
        +\sum_{q\in T}z_q P_{J,T\setminus\{q\}}(x,y;z_{T\setminus\{q\}}).
\end{equation}
Fix a coefficient $x^ry^sz^\eta$ and put $h=r+s+2$. By Lemma~\ref{lem:one-marked}, the one-marked part contributes $U_h(\eta)-L_h(\eta)$, where $U_h(\eta)\ge0$.

Now consider a summand $z_qP_{J,T\setminus\{q\}}$. It can contribute to $z^\eta$ only if $\eta_q=1$ and $\eta_{T\setminus\{q\}}\in\cB_J(T\setminus\{q\})$. In that case Lemma~\ref{lem:two-marked} gives the contribution
\begin{equation}\label{eq:P-contribution-local}
\begin{aligned}
        &\ind{h\le J+1}M(h-1,1,\eta_{T\setminus\{q\}}) \\
        &\quad +\ind{r\le J-1}\ind{s\le J-1}\ind{h\ge J+2}
        (h-J)M(J,h-J,\eta_{T\setminus\{q\}}).
\end{aligned}
\end{equation}
The first term in \eqref{eq:P-contribution-local} cancels exactly the corresponding summand of $L_h(\eta)$ when $h\le J$. More explicitly, for each $q\in T$ the lower-boundary summand to be canceled is
\begin{equation}\label{eq:local-cancellation-indicators}
        \ind{h\le J}\ind{\eta_q=1}
        \ind{\eta_{T\setminus\{q\}}\in\cB_J(T\setminus\{q\})}
        M(h-1,1,\eta_{T\setminus\{q\}}),
\end{equation}
and the first term in \eqref{eq:P-contribution-local}, together with the same conditions $\eta_q=1$ and $\eta_{T\setminus\{q\}}\in\cB_J(T\setminus\{q\})$, supplies precisely this quantity whenever $h\le J$. If $h=J+1$, the same first term remains as a nonnegative contribution because no lower-boundary deficit is present; if $h>J+1$, it is absent. The second term in \eqref{eq:P-contribution-local} is always nonnegative. Thus, after the lower-boundary cancellations, every remaining summand in the coefficient of $x^ry^sz^\eta$ in $Q_{J,T}$ is nonnegative. This proves coefficientwise nonnegativity.

Nontriviality is immediate. If $T=\varnothing$, then
\[
        Q_{J,\varnothing}(x,y)=(J+1)\sum_{r=0}^{J-1}x^{J-1-r}y^r.
\]
If $T\ne\varnothing$, the coefficient of $x^r y^s\prod_{q\in T}z_q$ with $r+s=J-1$ receives the positive upper-boundary contribution $(J+1)M(J,1,\ldots,1)$ from $U_{J+1}$, and hence $Q_{J,T}\ne0$.
\end{proof}

\begin{remark}
For $T=\varnothing$, Lemma~\ref{lem:local} gives
\[
        Q_{J,\varnothing}(x,y)=(J+1)\sum_{r=0}^{J-1}x^{J-1-r}y^r.
\]
This is exactly the $N=2$ local dissipation behind $1-p^j-(1-p)^j$. For larger $T$, the terms $z_qR_{J,\{x,y\}\cup(T\setminus\{q\})}$ in \eqref{eq:FJT} are precisely what cancel the lower-boundary deficits of the one-marked terms. This is the algebraic form of cumulative cancellation.
\end{remark}

\section{Radial derivative and positive kernels}\label{sec:proof-main}

Let
\[
        p_i(\theta)=\frac1N+\theta h_i,
        \qquad \sum_{i=1}^Nh_i=0,
        \qquad 0<\theta\le1.
\]
For a fixed subset $B$, put
\[
        \alpha_i=\frac{p_i(\theta)}{p_B(\theta)},
        \qquad p_B(\theta)=\sum_{i\in B}p_i(\theta).
\]
Then
\begin{equation}\label{eq:alpha-prime}
        \alpha_i'
        =\frac{|B|\alpha_i-1}{N\theta p_B}.
\end{equation}
Indeed, if $H_B=\sum_{i\in B}h_i$, then $p_B=|B|/N+\theta H_B$, and
\[
        \alpha_i'
        =\frac{h_i p_B-p_iH_B}{p_B^2}
        =\frac{Np_Bh_i-Np_iH_B}{N p_B^2}
        =\frac{|B|p_i-p_B}{N\theta p_B^2}
        =\frac{|B|\alpha_i-1}{N\theta p_B}.
\]
We shall use the following pair-polarization identity. For every differentiable function $\Phi$ on the simplex over $B$,
\begin{equation}\label{eq:pair-identity}
        \sum_{i\in B}\partial_i\Phi(\alpha)(|B|\alpha_i-1)
        =\sum_{a<b\in B}(\alpha_a-\alpha_b)(\partial_a\Phi-\partial_b\Phi).
\end{equation}
It follows from the elementary identity
\begin{equation}\label{eq:pair-polarization}
        \sum_{a<b}(x_a-x_b)(y_a-y_b)
        =|B|\sum_i x_iy_i-\left(\sum_i x_i\right)\left(\sum_i y_i\right),
\end{equation}
with $x_i=\alpha_i$ and $y_i=\partial_i\Phi(\alpha)$, using $\sum_i\alpha_i=1$.

We next make explicit how the local polynomial $F_{J,T}$ arises from the pair derivative of $\Phi_{m,j}$.

\begin{lemma}[Pair decomposition of the normalized derivative]\label{lem:pair-decomp}
Let $m\ge1$, let $B$ be finite with $|B|\ge m+1$, and let $a,b\in B$ be distinct. For $\alpha$ in the simplex over $B$,
\begin{equation}\label{eq:local-pair-diss}
        -\frac{\partial_a\Phi_{m,j}(\alpha)-\partial_b\Phi_{m,j}(\alpha)}{\alpha_a-\alpha_b}
        =\sum_{\substack{T\subseteq B\setminus\{a,b\}\\ |T|=m-1}}
        Q_{J,T}(\alpha_a,\alpha_b;\alpha_T).
\end{equation}
The quotient is interpreted by polynomial continuation when $\alpha_a=\alpha_b$.
\end{lemma}

\begin{proof}
Write $x=\alpha_a$, $y=\alpha_b$, and $z_r=\alpha_r$ for $r\in B\setminus\{a,b\}$. Terms in \eqref{eq:Phi-def} whose marked set $S$ contains neither $a$ nor $b$ do not contribute to $\partial_a-\partial_b$.

First suppose $S$ contains exactly one of $a,b$. For every $T\subseteq B\setminus\{a,b\}$ with $|T|=m-1$, the two marked sets $\{a\}\cup T$ and $\{b\}\cup T$ contribute
\[
        (1-z_T-x)R_{J,\{x\}\cup T}(x,z)
        +(1-z_T-y)R_{J,\{y\}\cup T}(y,z).
\]
These are the first two terms in $F_{J,T}$.

Now suppose $S$ contains both $a$ and $b$; this case is absent when $m=1$. Write $S=\{a,b\}\cup L$, where $|L|=m-2$. The contribution of this marked set to $\Phi_{m,j}$ is
\[
        (1-x-y-z_L)R_{J,\{x,y\}\cup L}(x,y,z_L).
\]
In the difference $\partial_x-\partial_y$, the derivatives of the coefficient $1-x-y-z_L$ cancel. Hence this marked set contributes
\[
        (1-x-y-z_L)(\partial_x-\partial_y)
        R_{J,\{x,y\}\cup L}(x,y,z_L).
\]
All partial derivatives in this lemma are ambient derivatives. The simplex constraint is used only after those derivatives have been taken. The next step uses that constraint and is therefore an identity after evaluation at the given simplex point $\alpha$, not an ambient polynomial identity in free variables. At that point,
\[
        1-x-y-z_L
        =\sum_{q\in B\setminus(\{a,b\}\cup L)}z_q.
\]
For each such $q$, put $T=L\cup\{q\}$. Then $|T|=m-1$, $q\in T$, and the corresponding evaluated contribution is
\[
        z_q(\partial_x-\partial_y)
        R_{J,\{x,y\}\cup(T\setminus\{q\})}(x,y,z_{T\setminus\{q\}}),
\]
which is the derivative of the third term in $F_{J,T}$. Combining the exactly-one terms with these evaluated both-marked terms gives, at the point $\alpha$,
\[
        (\partial_a-\partial_b)\Phi_{m,j}(\alpha)
        =\sum_{\substack{T\subseteq B\setminus\{a,b\}\\ |T|=m-1}}
        (\partial_x-\partial_y)F_{J,T}(x,y;z)\big|_{x=\alpha_a,\,y=\alpha_b,\,z=\alpha_T}.
\]
Dividing by $-(\alpha_a-\alpha_b)$ and using \eqref{eq:QJT} proves \eqref{eq:local-pair-diss}; the case $\alpha_a=\alpha_b$ follows by continuity.
\end{proof}

By Lemma~\ref{lem:local}, every polynomial $Q_{J,T}$ in \eqref{eq:local-pair-diss} is coefficientwise nonnegative and nonzero. Combining \eqref{eq:alpha-prime}, \eqref{eq:pair-identity}, and \eqref{eq:local-pair-diss}, we obtain
\begin{equation}\label{eq:local-derivative}
\begin{aligned}
        \frac{\dd}{\dd\theta}\Phi_{m,j}(\alpha^B)
        =-&\frac{1}{N\theta p_B}
        \sum_{a<b\in B}(\alpha_a-\alpha_b)^2 \\
        &\times
        \sum_{\substack{T\subseteq B\setminus\{a,b\}\\ |T|=m-1}}
        Q_{J,T}(\alpha_a,\alpha_b;\alpha_T).
\end{aligned}
\end{equation}

Now insert \eqref{eq:local-derivative} into the alternating expansion \eqref{eq:subset-expansion} and group the result by an unordered pair $\{a,b\}$. Since
\[
        \alpha_a-\alpha_b=\frac{p_a-p_b}{p_B},
\]
each monomial
\[
        c\,\alpha_a^r\alpha_b^s\prod_{i\in T}\alpha_i^{\nu_i}
        \qquad (c\ge0)
\]
of a polynomial $Q_{J,T}$ contributes a term of the form
\begin{equation}\label{eq:monomial-contribution}
        -\frac{1}{N\theta}(p_a-p_b)^2c
        \sum_{R\subseteq C}(-1)^{|R|}
        \frac{p_a^rp_b^s\prod_{i\in T}p_i^{\nu_i}}
        {(p_a+p_b+p_T+p_R)^L},
\end{equation}
where
\[
        C=[N]\setminus(\{a,b\}\cup T),
\]
and $L=r+s+\sum_{i\in T}\nu_i+3$ is a positive integer. The three additional powers in $L$ come from the factor $1/p_B$ in \eqref{eq:local-derivative} and from $(\alpha_a-\alpha_b)^2=(p_a-p_b)^2/p_B^2$. The sign $(-1)^{|R|}$ is exactly the sign in \eqref{eq:subset-expansion}, because $|B|=m+1+|R|$ when $B=\{a,b\}\cup T\cup R$.

The alternating denominator sum in \eqref{eq:monomial-contribution} is strictly positive. Indeed,
\begin{equation}\label{eq:laplace-positive}
\begin{aligned}
        \sum_{R\subseteq C}(-1)^{|R|}
        \frac{1}{(p_a+p_b+p_T+p_R)^L}
        &=\frac{1}{(L-1)!}
        \int_0^\infty t^{L-1}\e^{-(p_a+p_b+p_T)t}
        \prod_{\ell\in C}(1-\e^{-p_\ell t})\dd t  \\
        &>0.
\end{aligned}
\end{equation}
The strict positivity uses $p_i>0$ for every $i$; if $C=\varnothing$, the empty product is one.

The preceding calculation gives the following explicit positive-kernel form of the radial derivative. This is the main structural certificate behind the theorem.

\begin{theorem}[Positive-kernel radial derivative]\label{thm:positive-kernel}
Let $N\ge2$, $j\ge2$, and $1\le m\le N-1$. Along every ray
\[
        p_i(\theta)=\frac1N+\theta h_i,
        \qquad \sum_{i=1}^Nh_i=0,
        \qquad 0<\theta\le1,
\]
inside the positive simplex, there exist kernels $K_{ab}^{(m,j)}(p(\theta))>0$ such that
\begin{equation}\label{eq:factorial-derivative-factor}
        \frac{\dd}{\dd\theta}C_m^{(j)}(p(\theta))
        =-\frac{1}{N\theta}
        \sum_{1\le a<b\le N}(p_a(\theta)-p_b(\theta))^2K_{ab}^{(m,j)}(p(\theta)).
\end{equation}
\end{theorem}

\begin{proof}
The identity follows by inserting \eqref{eq:local-derivative} into the alternating expansion \eqref{eq:subset-expansion}, grouping the resulting terms by the unordered pair $\{a,b\}$, and using the monomial contribution \eqref{eq:monomial-contribution}. To make the kernel explicit, fix $a<b$. For every set
\[
        T\subseteq [N]\setminus\{a,b\},\qquad |T|=m-1,
\]
write
\[
        Q_{J,T}(x,y;z_T)
        =\sum_{r,s,\nu}c_{r,s,\nu}^{T}x^ry^s z_T^\nu,
        \qquad c_{r,s,\nu}^{T}\ge0,
\]
where the sum is finite and ranges over the monomials of $Q_{J,T}$. For such a monomial put
\[
        C=[N]\setminus(\{a,b\}\cup T),
        \qquad L=r+s+|\nu|+3.
\]
The corresponding contribution to the coefficient of $-(p_a-p_b)^2/(N\theta)$ is
\[
        c_{r,s,\nu}^{T}p_a^rp_b^s\prod_{i\in T}p_i^{\nu_i}
        \sum_{R\subseteq C}(-1)^{|R|}
        \frac{1}{(p_a+p_b+p_T+p_R)^L}.
\]
Thus $K_{ab}^{(m,j)}(p)$ is obtained by summing these displayed quantities over all admissible $T$ and over all monomials of $Q_{J,T}$.

Each summand is nonnegative. Indeed $c_{r,s,\nu}^{T}\ge0$ by Lemma~\ref{lem:local}, the monomial value in the positive variables $p_i$ is nonnegative, and the alternating denominator sum is strictly positive by \eqref{eq:laplace-positive}. Moreover the whole kernel is strictly positive. Since $m\le N-1$, for every pair $a,b$ there exists at least one set $T\subseteq[N]\setminus\{a,b\}$ with $|T|=m-1$. For any such $T$, Lemma~\ref{lem:local} says that $Q_{J,T}$ is not the zero polynomial and is coefficientwise nonnegative, so it has at least one positive coefficient $c_{r,s,\nu}^{T}>0$. Because all $p_i$ are positive and the associated Laplace sum is strictly positive, this monomial contributes strictly positively to $K_{ab}^{(m,j)}(p)$. Hence $K_{ab}^{(m,j)}(p)>0$ for every pair $a<b$.
\end{proof}

\begin{theorem}[Radial shifted factorial-moment monotonicity]\label{thm:factorial}
For every $N\ge2$, $j\ge2$, and $1\le m\le N-1$, and for every positive nonuniform probability vector $p$,
\[
        \theta\mapsto C_m^{(j)}(u+\theta(p-u))
\]
is strictly decreasing on $(0,1]$.
\end{theorem}

\begin{proof}
Apply Theorem~\ref{thm:positive-kernel} to the ray $p(\theta)=u+\theta(p-u)$. Because $p\ne u$, the vector $p(\theta)$ is nonuniform for every $0<\theta\le1$. Hence, for each such $\theta$, at least one squared difference $(p_a(\theta)-p_b(\theta))^2$ is positive. Since every kernel in \eqref{eq:factorial-derivative-factor} is strictly positive,
\begin{equation}\label{eq:Cm-strict}
        \frac{\dd}{\dd\theta}C_m^{(j)}(p(\theta))<0
        \qquad (0<\theta\le1).
\end{equation}
Integrating this strict inequality over any interval $[\theta_1,\theta_2]\subset(0,1]$, with $\theta_1<\theta_2$, gives
\[
        C_m^{(j)}(p(\theta_2))<C_m^{(j)}(p(\theta_1)).
\]
Thus $\theta\mapsto C_m^{(j)}(u+\theta(p-u))$ is strictly decreasing on $(0,1]$.
\end{proof}

The main theorem follows immediately.

\begin{proof}[Proof of Theorem~\ref{thm:main}]
Write $q=z-1$ and $p(\theta)=u+\theta(p-u)$. By Proposition~\ref{prop:pgf-identity},
\[
        \E_{p(\theta)}z^{U_j^N}
        =z\sum_{m=0}^{N-1}q^m C_m^{(j)}(p(\theta)).
\]
Here $C_0^{(j)}\equiv1$. If $z>1$, then $q>0$, and Theorem~\ref{thm:factorial} gives
\[
        \frac{\dd}{\dd\theta}\E_{p(\theta)}z^{U_j^N}<0
\]
whenever $p(\theta)$ is nonuniform. Since $p\ne u$ implies $p(\theta)\ne u$ for every $\theta>0$, the displayed derivative is strictly negative on $(0,1]$. The left side is continuous at $\theta=0$; for instance, this follows from the finite normalized subset expansion. Integrating the strict derivative on $[\varepsilon,1]$ and then letting $\varepsilon\downarrow0$ gives strict radial PGF monotonicity from the uniform endpoint:
\[
        \E_p z^{U_j^N}<\E_u z^{U_j^N}\qquad (z>1).
\]

The binomial-moment monotonicity is obtained from the coefficient theorem, not from differentiating a pointwise PGF inequality. For $1\le m\le N$,
\[
        \binom{U_j^N}{m}
        =\binom{U_j^N-1}{m}+\binom{U_j^N-1}{m-1}.
\]
Therefore
\[
        \E_p\binom{U_j^N}{m}
        =C_m^{(j)}(p)+C_{m-1}^{(j)}(p),
\]
where $C_N^{(j)}\equiv0$ and $C_0^{(j)}\equiv1$. Theorem~\ref{thm:factorial} shows that this expression is strictly decreasing in $\theta$ on $(0,1]$: for $m=1$, the strict term is $C_1^{(j)}$; for $2\le m\le N-1$, both displayed terms are monotone and at least one is strict; and for $m=N$, the strict term is $C_{N-1}^{(j)}$. Continuity at $\theta=0$ then gives the endpoint inequality against the uniform distribution. For $m>N$, both binomial coefficients vanish identically.
\end{proof}

\begin{definition}[Finite absolutely-monotone order]
For random variables $X,Y$ supported on $\{0,1,\ldots,N\}$, write
$X\le_{\rm abm}Y$ if
\[
        \E f(X)\le \E f(Y)
\]
for every function $f:\{0,1,\ldots,N\}\to\R$ with binomial-basis expansion
\[
        f(k)=\sum_{m=0}^N a_m\binom{k}{m},
        \qquad a_m\ge0\quad(m\ge1).
\]
Equivalently, $X\le_{\rm abm}Y$ iff
\[
        \E\binom{X}{m}\le \E\binom{Y}{m}
        \qquad(1\le m\le N).
\]
\end{definition}

\begin{corollary}[Radial absolutely-monotone order]\label{cor:abm}
Fix $N\ge2$, $j\ge2$, and a positive nonuniform probability vector $p$. Let
\[
        p(\theta)=u+\theta(p-u),\qquad 0\le\theta\le1.
\]
If $0\le\theta_1<\theta_2\le1$, then
\[
        U_{j,p(\theta_2)}^N\le_{\rm abm}U_{j,p(\theta_1)}^N.
\]
Equivalently, for every function
\[
        f(k)=\sum_{m=0}^N a_m\binom{k}{m},
        \qquad a_m\ge0\quad(m\ge1),
\]
one has
\[
        \E_{p(\theta_2)}f(U_j^N)
        \le
        \E_{p(\theta_1)}f(U_j^N).
\]
The inequality is strict whenever at least one coefficient $a_m$ with $m\ge1$ is positive.
\end{corollary}

\begin{proof}
The binomial-moment part of Theorem~\ref{thm:main} gives, for every $1\le m\le N$,
\[
        \E_{p(\theta_2)}\binom{U_j^N}{m}
        <
        \E_{p(\theta_1)}\binom{U_j^N}{m}
\]
whenever $0\le\theta_1<\theta_2\le1$ and $p\ne u$. Taking nonnegative linear combinations of these inequalities in the binomial basis proves the claim. The coefficient $a_0$ contributes the same constant to both sides.
\end{proof}

\begin{remark}[Open problems: stochastic and increasing-convex orders]\label{rem:open-orders}
The natural next strengthenings would be the radial stochastic order
\[
        U_{j,p(\theta_2)}^N\le_{\rm st}U_{j,p(\theta_1)}^N
        \qquad(0\le\theta_1<\theta_2\le1)
\]
and the radial increasing-convex order
\[
        U_{j,p(\theta_2)}^N\le_{\rm icx}U_{j,p(\theta_1)}^N.
\]
The present factorial-moment theorem does not imply either order. Tail indicators and stop-loss functions have alternating binomial-basis expansions, so these stronger orders would require new positive tail-kernel or stop-loss-kernel formulas. We leave these as open problems.
\end{remark}

\section{Checks and special cases}

\subsection{Recovery of the first moment}

For $m=1$, Theorem~\ref{thm:factorial} says that
\[
        \E_p(U_j^N-1)
\]
is strictly radially maximized at uniform. Since the final coupon contributes one deterministic empty space, this is equivalent to radial monotonicity of $\E_pU_j^N$. Thus Theorem~\ref{thm:main} strengthens the finite-$N$ maximum principle for the expected number of empty spaces in the sibling's album.

\subsection{The oldest sibling}

For $j=2$, $A_2(x)=x$ and
\[
        \Phi_{m,2}(\alpha)=(m+1)!e_{m+1}(\alpha).
\]
In this case \eqref{eq:factorial-derivative-factor} specializes to
\[
        \frac{\dd}{\dd\theta}C_m^{(2)}(p(\theta))
        =-\frac{(m+1)!}{N\theta}
        \sum_{a<b}(p_a-p_b)^2K_{ab}^{(m)}(p),
\]
where
\[
        K_{ab}^{(m)}(p)
        =\frac{1}{(m+1)!}\int_0^\infty
        t^{m+1}\e^{-(p_a+p_b)t}
        D_{m-1}^{ab}(t)\dd t
\]
and
\[
        D_{m-1}^{ab}(t)
        =\sum_{\substack{L\subseteq[N]\setminus\{a,b\}\\ |L|=m-1}}
        \left(\prod_{\ell\in L}p_\ell\e^{-p_\ell t}\right)
        \prod_{r\notin L\cup\{a,b\}}(1-\e^{-p_rt}).
\]
This gives a particularly transparent proof of full PGF radial monotonicity for $U_2^N$ on the range $z>1$. As a boundary sanity check, the factorial-moment conclusion with $j=2$ and $m=N$ says
\[
        \Pbb_p\{U_2^N=N\}\le \Pbb_u\{U_2^N=N\}.
\]
The event $U_2^N=N$ means that all coupon types have appeared exactly once when the main collector finishes, so its probability is $N!\prod_{i=1}^N p_i$; the displayed inequality is therefore the arithmetic--geometric mean inequality $\prod_i p_i\le N^{-N}$.

\subsection{Why cumulative empty-space events are necessary}

The exact-count pieces in $A_j$ need not be radially extremal. Already for $N=2$, an exact-count contribution of order $r$ has the form
\[
        p(1-p)^r+(1-p)p^r,
\]
which is not maximized at $p=1/2$ for all $r$. The theorem works because
\[
        A_j(x)=x+\frac{x^2}{2!}+\cdots+\frac{x^{j-1}}{(j-1)!}
\]
is cumulative. The finite cancellation in Lemma~\ref{lem:local} is precisely the cancellation that fails for individual exact-count components.

\section{Conclusion}

Uniform probabilities do not merely maximize the expected number of empty spaces in the siblings problem. Along every nonconstant ray from the uniform vector, they maximize the empty-space probability generating function for every $z>1$ and all binomial moments of the number of empty spaces in the sibling's album. Equivalently, the laws decrease away from uniform in the finite absolutely-monotone/binomial-transform order. In the complementary range $0<z<1$, the transform moves in the opposite direction: moving away from uniform strictly increases $\E z^{U_j^N}$, giving radial Laplace-transform order. The right-PGF and factorial-moment proof is finite-dimensional and exact: after Poissonization and a marked coefficient expansion, the radial derivative of each factorial coefficient factors into a negative sum of squared pair differences times strictly positive kernels. The Laplace-transform proof is probabilistic, using a Gamma-mixture race representation of the same Poissonized PGF identity.

\end{document}